\documentclass{amsart}

 \usepackage{latexsym}

 \usepackage{amssymb}

 \usepackage{amsfonts}

 \usepackage{amsmath}

 \usepackage{amsthm}
\usepackage{float}

\usepackage{graphics}

\usepackage[pdftex]{graphicx}

\newtheorem{theorem}{Theorem}[section]
\newtheorem{corollary}[theorem]{Corollary}
\newtheorem{lemma}[theorem]{Lemma}
\newtheorem{proposition}[theorem]{Proposition}
\newtheorem{example}[theorem]{Example}

\newtheorem{remark}[theorem]{Remark}
\newtheorem{definition}[theorem]{Definition}
\newtheorem{problem}[theorem]{Problem}

\newtheorem{con}{Conjecture}

\def\Sm0#1{{{\rm GL}}(1,#1)}

\begin{document}

\title[ Conjectures on the normal covering number ]
{Conjectures on the normal covering number of the finite symmetric and alternating groups}
\author[D.~Bubboloni]{Daniela Bubboloni}
\address{Daniela Bubboloni, Dipartimento di Scienze per l'Economia e l'Impresa,\newline
University of Firenze, \newline Via delle Pandette 9-D6,  50127 Firenze, Italy}
\email{daniela.bubboloni@unifi.it}

\author[C. E. Praeger]{Cheryl E. Praeger}
\address{Cheryl E. Praeger, Centre for Mathematics of Symmetry and Computation,\newline
School of Mathematics and Statistics,\newline
The University of Western Australia,\newline
 Crawley, WA 6009, Australia\newline
 Also affiliated with King Abdulaziz University, Jeddah, Saudi Arabia}\email{Cheryl.Praeger@uwa.edu.au}

\author[P. Spiga]{Pablo Spiga}
\address{Pablo Spiga,
Dipartimento di Matematica e Applicazioni,\newline
 University of Milano-Bicocca,\newline Via Cozzi 53, 20125 Milano, Italy
}
\email{pablo.spiga@unimib.it}

\begin{abstract} Let $\gamma(S_n)$ be the minimum number of proper subgroups
$H_i,\ i=1, \dots, l $ of the symmetric group $S_n$ such that each element in $S_n$
 lies in some conjugate of one of the $H_i.$ In this paper we
conjecture that $$\gamma(S_n)=\frac{n}{2}\left(1-\frac{1}{p_1}\right)
\left(1-\frac{1}{p_2}\right)+2,$$ where $p_1,p_2$ are the two smallest primes
in the factorization of $n\in\mathbb{N}$ and $n$ is neither a prime power nor
a product of two primes. Support for the conjecture is given by a previous result
for $n=p_1^{\alpha_1}p_2^{\alpha_2},$ with $(\alpha_1,\alpha_2)\neq (1,1)$.
We give further evidence by confirming the conjecture for integers
of the form $n=15q$ for an infinite set of primes $q$, and by reporting on a
{\tt Magma} computation. We make a similar conjecture for $\gamma(A_n)$, when $n$ is even,
and provide a similar amount of evidence.

 \bigskip


\end{abstract}
\maketitle
\section{Introduction}
Let $G$ be the symmetric group $S_n$ or the alternating group $A_n$ of degree $n\in\mathbb{N},$ acting naturally on the set $\Omega=\{1,\dots,n\}.$
In this paper we study the optimum way of covering the elements of $G$ by conjugacy classes of proper subgroups. If $H_1,\ldots,H_l$, with $l\in\mathbb{N},l\geq 2$
 are pairwise
non-conjugate proper subgroups of $G$ such that
$$
G=\bigcup_ {i=1}^l\bigcup_{g\in G}H_i^g,
$$
we say that $\Delta=\{H_i^g\,|\,1\leq i\leq l, g\in G\}$ is a {\em
  normal covering} of $G$ and that $\delta=\{H_1,\dots,H_l\}$ is a
{\em basic set} for $G$ which generates $\Delta.$ We call the elements
of $\Delta$ the {\em components}  and
the elements of $\delta$ the {\em basic components} of the normal covering.
The minimum cardinality $|\delta|$ of a basic set $\delta$ for $G$
is called the \emph{normal covering number} of $G$, denoted $\gamma(G)$
and a basic set of size $\gamma(G)$ is called a {\it minimal basic set}.
Note that the parameters $\gamma(S_n)$, $\gamma(A_n)$ are defined only for
$n\geq 3$ and $n\geq 4,$ respectively. More generally $\gamma(G)$ is defined for each non-cyclic
group $G$, since cyclic groups cannot be
covered by proper subgroups.

The first two authors introduced the parameter $\gamma(G)$ in
\cite{BP}, and proved there that $a\varphi(n)\leq \gamma (G)\leq
bn$, where $\varphi(n)$ is the Euler totient function and
$a, b$ are positive real constants depending on
whether $G$ is alternating or symmetric and whether $n$ is even or odd
(see~\cite[Theorems~A and~B]{BP}).

Recently, in ~\cite[Theorem 1.1]{BPS}, we established a lower bound linear
in $n$, given by $cn \leq  \gamma(G) \leq \frac{2}{3} n,$
where $c$ is a positive real constant, and in \cite[Remark 6.5]{BPS}, which describes
a general method for obtaining explicit values for $c$, it was shown that in the case where $G=S_n$ with $n$ even, the constant $c$ can be taken as $0.025$ for $n>792000$. This settled the question as to whether $\gamma(G)$ grows linearly in $n$, rather than growing more slowly as a constant times $\varphi(n)$.
Proof of the improved lower bound in \cite{BPS} relies on certain number theoretic results (see \cite{BLS}), and the value of $c$ obtained is unrealistically small
because of the many approximations needed first to obtain and next to apply those results.
Thus, despite the innovative methods used to achieve the definitive result that  $\gamma(S_n)$ and $\gamma(A_n)$ grow linearly with $n$, we still do not know, and we wish to know, optimum values for the constant $c$, both for $\gamma(S_n)$ and $\gamma(A_n)$.

We know quite a lot (but not everything) about these parameters in the case where $n$ is a prime power, and we summarise these results in Remark~\ref{rem1}(c).

Our interest in this paper is in the case where $n$ has at least two distinct prime divisors, that is to say, $n$ is of the form

\begin{equation}\label{factors}
n=p_1^{\alpha_1}\cdots p_r^{\alpha_r}
\end{equation}
where $r\geq2$, the $p_i$ are primes, $\alpha_i\in\mathbb{N}$, and $p_i<p_j$ for $i<j$.
We make first two conjectures about the normal covering numbers of $S_n$ and $A_n$ for such values of $n$. We believe that the value of $\gamma(S_n)$, for all such $n$, as well as
the value of $\gamma(A_n)$ when $n$ is even, is strongly connected to the following quantity:

\begin{equation}\label{g}
g(n):=\frac{n}{2}\left(1-\frac{1}{p_1}\right)\left(1-\frac{1}{p_2}\right)+2
\end{equation}
where $n, p_1, p_2$ are as in \eqref{factors}.
Note that $g(n)$ is an integer since  $g(n)-2$ is equal to the number of natural numbers  less than $n/2$ and not divisible by either $p_1$ or $p_2$(see Proposition \ref{arithm} iii)). 
In the case where $r=2$ in \eqref{factors}, a connection between $\gamma(G)$ and $g(n)$ emerged from results in \cite{BPS}, which are summarised in Remark~\ref{rem1}~(d). These results, and other evidence, motivate the following conjectures (which will also appear in the forthcoming edition of the Kourovka Notebook \cite{K}).


 \begin{con}\label{conjecture-sym}
 Let $n$ be as in \eqref{factors} with $r\geq2$ and $n\neq p_1p_2.$ Then $\gamma(S_n)=g(n)$.
\end{con}

 \begin{con}\label{conjecture-alt}
Let $n$ be as in \eqref{factors} with $n$ even, $r\geq2$, and $n\neq 2p_2, 12.$ Then  $\gamma(A_n)=g(n)$.
\end{con}

As we note in Remark~\ref{rem1}~(d), if $r=2$ then Conjecture \ref{conjecture-alt}
is true, and also Conjecture \ref{conjecture-sym}
is true for odd $n$ (but open for even $n$). Moreover, extensive computation performed with {\tt Magma} \cite{Magma} 
shows that both conjectures hold for the first several hundred values of $n$.

That $g(n)$ is a common upper bound for $\gamma(S_n)$ and  $\gamma(A_n)$ for all $n\in\mathbb{N},$ with $r\geq2,$ is not difficult to prove (see Proposition~\ref{upper}).
The issue is to show that $g(n)$ is also a lower bound. We provide in Theorem~\ref{thm1} further evidence
for the truth of Conjectures~\ref{conjecture-sym} and~\ref{conjecture-alt}, for infinite families of integers
$n$ with three distinct prime divisors (see also Corollary~\ref{support}).

\begin{theorem}\label{thm1}
\begin{enumerate}
\item[(a)] Let $n=15q$, where $q$ is an odd prime such that $q\equiv 2 \pmod{15}$ and $q\not\equiv 12 \pmod{13}$.
Then $\gamma(S_{n})=g(n)=4q+2.$
\item[(b)] Let $n=6q$, where $q$ is a prime, $q\geq 11$. Then $\gamma(A_{n})=g(n)=q+2.$
\end{enumerate}
\end{theorem}

The role of the function $g(n)$ in the study of normal coverings of finite groups seems to go beyond the symmetric and alternating case:  Britnell
 and Mar\'{o}ti in \cite{BM} have recently shown that, if $n$ is as in (\ref{factors}), with $r\geq 2,$ and $q$ is a prime power, then for all linear
  groups $G$ with $SL_n(q)\leq G\leq GL_n(q),$ the upper bound $\gamma(G)\leq g(n)$ 
holds, with equality for $r=2$ and for other infinite families of $n$-dimensional classical groups with $r\geq 3.$

We make now several more remarks about  Conjectures \ref{conjecture-sym} and \ref{conjecture-alt}.

\begin{remark}\label{rem2}{\rm
(a) Observe that 
$\frac{1}{2}\left(1-\frac{1}{p_1}\right)\left(1-\frac{1}{p_2}\right)\geq \left(1-\frac{1}{2}\right)\left(1-\frac{1}{3}\right)=\frac{1}{6}$. Thus,
 for all $n\in\mathbb{N}$ with $r\geq 2$ and $n\neq p_1p_2,$
 Conjecture \ref{conjecture-sym}  implies $\gamma(S_n)\geq \frac{n}{6}+2,$ and thus $\gamma(S_n)\geq cn$ with $c=\frac{1}{6},$ a value much larger than
  the value $0.025$  given in \cite[Remark 6.5]{BPS}. Moreover, if Conjecture~\ref{conjecture-sym} is true,
 then $c=\frac{1}{6}$ is the largest constant such that $\gamma(S_n)\geq cn$ for all $n\in\mathbb{N}$. Namely Conjecture~\ref{conjecture-sym}
  would imply that $\gamma(S_n)=g(n) =\frac{n}{6}+2$ for all multiples $n$ of $6$, greater than $6.$

(b) The case $n=12$ is a genuine exception in Conjecture \ref{conjecture-alt} because $\gamma(S_{12})=4=g(12),$
while $\gamma(A_{12})=3$ (see Lemma \ref{12}). We note that  Lemma \ref{12} corrects \cite[Proposition 7.8(c)]{BP}
which is incorrect just in the case $n=12$; this error is repeated in \cite[Corollary 7.10]{BP} (which
asserts that $\gamma(A_{12})=4$). We give therefore a correct version of \cite[Table 3]{BP}.
{\openup 2pt
\begin{center}
 \begin{table}[H]
\begin{tabular}{l|llllllllll}\hline
$n$&3&4&5&6&7&8&9&10&11&12\\ \hline
$\gamma(S_n)$&2&2&2&2&3&3&4&3&   5&4\\ 
$\gamma(A_n)$&-&2&2&2&2&2&3&3&4&3\\  \hline
\end{tabular}
\caption{Values of $\gamma(S_n)$ and $\gamma(A_n)$ for small $n$.}\label{tbl:exact}
 \end{table}
\end{center}
}
(c) Conjecture~\ref{conjecture-alt} does not address the case of $G=A_n$ with $n$ odd. For $n$ an odd prime, we know (see Remark~\ref{rem1}~(c)) that the difference
between $\gamma(S_n)$ and $\gamma(A_n)$ is unbounded. If this were the case also for odd non-prime powers $n$, then these parameters could not be both equal to $g(n).$ Although our computer experiments indicate that $\gamma(A_n)\ne g(n)$ for some odd integers $n$,
we do not have a sufficiently good understanding to predict the general behaviour of $\gamma(A_n)$ for odd $n$.
}
\end{remark}

Our third conjecture relates to the form of a minimal basic set. We state it in Section~\ref{maximal} as Conjecture~\ref{conjecture-intr-sa}, since it uses notation introduced in that section. In particular we show, in Proposition~\ref{conjectures}, 
that Conjecture~\ref{conjecture-intr-sa} implies both Conjecture~\ref{conjecture-sym} for $n$ odd, and Conjecture~\ref{conjecture-alt}.
Finally, in this introductory section, we make a few general comments about the normal covering number.


\begin{remark}\label{rem1}{\rm
(a) When computing $\gamma(G)$,  we can always assume that the basic components in $\delta$ are maximal subgroups of $G$.

(b) If $\gamma(S_n)$ is realized by a basic set $\delta$ not
involving $A_n$, then $\gamma(A_n)\leq \gamma(S_n)$, since
in this case the set of subgroups $H\cap A_n$, for $H\in\delta$, forms a basic set for $A_n$.
It is not known whether the inequality
$\gamma(A_n)\leq \gamma(S_n)$ holds for all $n$, and this remains an open question.

(c) Suppose that $n=p^a$ for a prime $p$ and a positive integer $a$.
For $a=1$ we know that, for $p$ odd, $\gamma(S_p)= \frac{p-1}{2}$ (see \cite[Proposition 7.1]{BP}) while $\gamma(A_p)$ lies between $\lceil\frac{p-1}{4}\rceil$ and $\lfloor\frac{p+3}{3}\rfloor$ (see \cite[Remark 7.2]{BP}). In particular we have $\gamma(S_p)-\gamma(A_p)\geq \frac{p-9}{6}.$ 

Now consider $a\geq2$. Then $\gamma(S_n)=\frac{n}{2}(1-\frac{1}{p}) +1$ if $p$ is odd, and if $p=2$ then $\gamma(S_n)$ lies between $\frac{n+8}{12}$ and $\frac{n+4}{4}$ (see \cite[Proposition 7.5]{BP}). For alternating groups we have $\gamma(A_n)= \frac{n+4}{4}$ if $p=2$ and $n\ne8$ 
 and when $p$ is odd we only know that $\gamma(A_n)$ lies between  $\frac{n}{4}(1-\frac{1}{p})$ and $\frac{n}{2}(1-\frac{1}{p}) +1$.

(d) Suppose now that $n$ is as in \eqref{factors} with $r=2$. If $\alpha_1+\alpha_2=2$, that is to say, if  $n=p_1p_2$, then by \cite[Proposition 7.6]{BP},
$\gamma(S_{n})=g(n)-1$ if $n$ is odd, while $\gamma(A_{n})=g(n)-1$ if $n$ is even,
with $g(n)$ as in \eqref{g}.
Moreover, if $\alpha_1+\alpha_2\geq 3$, then
it was shown in  \cite[Proposition 7.8]{BP} that
$\gamma(S_{n})=g(n)$ if $n$ is odd, while $\gamma(A_{n})=g(n)$ if $n$ is even.
In particular, these results show that,  when $r=2$,  Conjecture~\ref{conjecture-alt} is true, and also Conjecture~\ref{conjecture-sym} is true for odd $n$.

}
\end{remark}

In the light of Remark~\ref{rem1}(b)--(d) we pose the following problems.

\begin{problem}\label{prob1}{\rm
Determine whether or not the inequality $\gamma(A_n)\leq \gamma(S_n)$ holds for all $n\geq4$. Moreover is it true that $\gamma(A_n)= \gamma(S_n)$  for all but finitely many even values of $n$? 
}
\end{problem}

\begin{problem}\label{prob2}{\rm
Determine the values of $\gamma(S_n)$ for $n=2^a\geq 16$, and of $\gamma(A_n)$ for all odd prime powers $n$. 
}
\end{problem}
\begin{problem}\label{prob3}{\rm
Determine the values of $\gamma(S_n)$ for $n=2^{\alpha_1}p_2^{\alpha_2}$, and of $\gamma(A_n)$ for odd integers $n=p_1^{\alpha_1}p_2^{\alpha_2}$, where $\alpha_1,\alpha_2\in\mathbb{N}.$
}
\end{problem}
\section{Notation and basic facts}
 \subsection{Arithmetic}\label{arit}
 In the study of the normal coverings of symmetric and alternating groups we often face some arithmetic questions. Thus it is natural to begin with a purely arithmetic section.
For $a,b\in\mathbb{R},$ with $a\leq b$, we use the usual notation $(a,b), [a,b),(a,b], [a,b]$ to denote the intervals.

Throughout this section, for $n\in\mathbb{N}$, assume the notation (\ref{factors}), with $r\in\mathbb{N},$
 that is write $n=p_1^{\alpha_1}\cdots p_r^{\alpha_r},$ with $p_i$ primes, $\alpha_i\in\mathbb{N}$, and $p_i<p_j$ for $i<j$. Note that we include here the possibility that $n$ is a prime power.

 \begin{definition}\label{arithm-def}  For each subset $K$ of $R:=\{1,\dots,r\}$, define the following subsets of $\mathbb{N}:$
$$
P_K:= \{x\in\mathbb{N} : 1\leq x\leq n,\  p_k\mid x \hbox{ for all}\  k\in K\}
$$
$$
P^K:=\{x\in\mathbb{N} : 1\leq x\leq n,\  p_k\nmid x \hbox{ for all}\  k\in K\}
$$
and the rational numbers:
$$
p_K:=\prod_{k\in K} \frac{1}{p_{k}}, \quad\quad p^K:=\prod_{k\in K} \left(1-\frac{1}{p_{k}}\right).
$$
For $I,J\subseteq R$ with $I\cap J=\varnothing,$  set 
$$
P_I^J:=P_I\cap P^J.
$$

 \end{definition}

Note that $P_I=P_I^{\varnothing}$ and $P^J= P_{\varnothing}^J$ and, as usual, we consider a product over the empty set to be equal to $1.$ From now on we 
use the notation introduced in Definition~\ref{arithm-def}, without further reference. For the purpose of this paper it is important to obtain the order
  of  $P_I^J\cap [1,n/2),$ because that allows us to give an interpretation of the quantity $g(n)$ as well as to count,
in many circumstances, various sets of partitions or sets of intransitive maximal subgroups of $S_n$. We start by finding the cardinality of $P_I^J.$

\begin{lemma}\label{general} Let $n,a,b,c\in\mathbb{N},$ with $n=abc$ and $\gcd(a,b)=1.$ 
 If $d\mid a$, then $$|\{x\in\mathbb{N}: 1\leq x\leq n,\  d\mid x,\  \gcd(x,b)=1\}|=\frac{n\,\varphi(b)}{bd}.$$ In particular, for all $I,J\subseteq R$ with $I\cap J=\varnothing,$ we have $|P_I^J|=n\,p_Ip^J$.
\end{lemma}
\begin{proof}  Let $n,a,b,c,d\in\mathbb{N},$ with $n=abc$, $\gcd(a,b)=1$ 
 and $d\mid a$. We
set $$X=\{x\in\mathbb{N}: 1\leq x\leq n,\  d\mid x,\  \gcd(x,b)=1\}$$    and  $$Y=\{ x\in\mathbb{N}: \hbox{there exists}\ u\in\mathbb{N}, u\leq n/d\   \hbox{with}\  x=du,\  \gcd(u,b)=1\}.$$
Since $d\mid a$ and $\gcd(a,b)=1$,  we have that $\gcd(d,b)=1.$ It follows that $X=Y.$  In particular $|X|$ is equal to the number of positive integers less than or equal to $\displaystyle{n/d=b\left(c\,\frac{a}{d}\right)}\in\mathbb{N}$ which are coprime to $b$. But, given any $k\in\mathbb{N},$  the number of positive integers coprime to $b$ and contained in the interval $[1,bk],$ is given by $k\varphi(b).$ Thus $|X|=\displaystyle{c\,\frac{a}{d}\varphi(b)=\frac{n\,\varphi(b)}{bd}}.$

Let now $I,J\subseteq R$ with $I\cap J=\varnothing$ and consider
 $a=\prod_{i\in I} p_i^{\alpha_i}, d=\prod_{i\in I} p_i, b=\prod_{j\in J} p_j^{\alpha_j}, c=\frac{n}{ab}.$ Since $I\cap J=\varnothing$, we have that $\gcd(a,b)=1$ and thus
  the above result  applies giving
 $$|P_I^J|=\frac{nb\prod_{j\in J}(1-\frac{1}{p_j})}{b\prod_{i\in I} p_i}=n\,p_Ip^J.$$

\end{proof}

Next we decide, for $n$ even, when the integer
 $n/2$ lies in $P_I^J.$ 
   \begin{lemma}\label{n:2} Let $n\in \mathbb{N}$ be even (that is, $p_1=2$),
and let $I,J\subseteq R$ with $I\cap J=\varnothing.$
  \begin{itemize}
  \item[i)] Suppose $\alpha_1=1.$ Then $n/2\in P_I^J$ if and only if $I\subseteq \{2,\dots,r\}$ and either $J=\{1\}$ or $J=\varnothing$.
    \item[ii)] Suppose $\alpha_1\geq 2.$ Then $n/2\in P_I^J$ if and only if $J=\varnothing.$
\end{itemize}

   \end{lemma}
   \begin{proof} i) Since $\alpha_1=1$ , we have that  $n/2=p_2^{\alpha_2}\cdots p_r^{\alpha_r}$ is odd. Suppose that $n/2\in P_I^J$. Then $1\notin I$ so
 $I\subseteq \{2,\dots,r\}.$ Moreover, if $j\in J$ then $p_j\nmid n/2=p_2^{\alpha_2}\cdots p_r^{\alpha_r}$ and thus the only possibility is $j=1$, so that $J\subseteq \{1\}.$ It follows that $J=\{1\}$ or $J=\varnothing$. Conversely, it is clear that, if  $I\subseteq \{2,\dots,r\}$ and either  $J=\{1\}$ or $J=\varnothing$, then $n/2\in P_I^J$

 ii) Since $\alpha_1\geq 2$, we have $n/2=2^{\alpha_1-1}\cdots p_r^{\alpha_r}$ even and divisible by the same prime factors as $n$. Thus if $n/2\in P_I^J,$  we necessarily have $J=\varnothing$. Conversely if $J=\varnothing$, then $n/2\in P_I^{\varnothing}=P_I$ holds for all $I\subseteq R$.
   \end{proof}

   \begin{proposition}\label{arithm}
Let $n\in\mathbb{N}$ and let $I,J\subseteq R$ with $I\cap J=\varnothing.$

   \begin{itemize}
 \item[i)] If $n$ is odd, then $|P_I^J\cap [1,n/2)|=\lfloor \frac{n}{2}p_Ip^J\rfloor.$
 \item[ii)] If $n$ is even (that is, $p_1=2$),  then we have the following:
 \begin{itemize}\item[a)]  Suppose $\alpha_1=1.$  If $J=\varnothing$ and $I\subseteq \{2,\dots,r\},$
then $|P_I^J\cap [1,n/2)|=\frac{n\,p_Ip^J}{2}-1, $ and otherwise  $|P_I^J\cap [1,n/2)|=\lfloor
\frac{n}{2}p_Ip^J\rfloor.$
\item[b)]  Suppose $\alpha_1\geq 2.$ If $J=\varnothing,$ then $|P_I^J\cap [1,n/2)|=\frac{n\,p_Ip^J}{2}-1,$
and otherwise $|P_I^J\cap [1,n/2)|= \frac{n}{2}p_Ip^J.$
\end{itemize}
\item[iii)] If $r\geq 2$, and $J=\{j_1,j_2\}\subseteq R$ with $j_1\neq j_2$, then
 $$|P^J\cap [1,n/2)|=\frac{n}{2}\left(1-\frac{1}{p_{j_1}}\right)\left(1-\frac{1}{p_{j_2}}\right)=g(n)-2.$$ In particular, for all $n\in\mathbb{N}$ with $r\geq2$,  $g(n)$ is a positive integer.
 \item[iv)] If $n>2$ then $|P^R\cap [1,n/2)|=\frac{\varphi(n)}{2}.$
\end{itemize}

 \end{proposition}

\begin{proof}  Let $A=P_I^J\cap (0,n/2)$ and $B=P_I^J\cap (n/2,n).$ Since, trivially, $A=P_I^J\cap [1,n/2),$ our aim is to compute $|A|.$
Consider the bijective map $f:\mathbb{Z}\to \mathbb{Z}$ given by $f(x)=n-x,$ for all $x\in\mathbb{Z}.$ Note that, since
for all $i\in R,$ we have $p_i\mid x$ if and only if $p_i\mid n-x,$ then $f(A)=B.$ In particular $|A|=|B|$.
Moreover, we have
$$
A\cup B\subseteq P_I^J\subseteq A\cup B\cup\{n/2,n\}
$$
and thus, by Lemma \ref{general}, observing that $A\cap B=\varnothing,$
$$
\frac{np_Ip^J}{2}-1\leq |A|\leq \frac{np_Ip^J}{2},
$$
where the bounds are not necessarily integers.

To decide the exact value of $|A|$ we must take in account when $n,n/2\in P_I^J$. It is
easy to see that $n\in  P_I^J$ if and only if $J=\varnothing$. For the
case of $n/2$ we invoke Lemma \ref{n:2}, and we distinguish several cases.

Assume first that $n$ is odd. If $J=\varnothing,$ we have $n\in P_I^J=P_I$ while $n/2$ is not an integer.
Hence in this case $2|A|+1=|P_I|=np_I$  which implies that $p_I$ is odd and
$|A|=\lfloor \frac{np_Ip^J}{2}\rfloor,$ since $p^J=1$.
If $J\neq \varnothing,$ then $n,n/2\notin  P_I^J$ and thus $2|A|=|P_I^J|=np_Ip^J,$ which
implies that $p_Ip^J$ is even and $|A|=\frac{np_Ip^J}{2}=\lfloor \frac{np_Ip^J}{2}\rfloor.$
In particular we have now proved part i), and we have also proved parts iii) and iv) for $n$ odd.

Next assume that $n$ is even but not divisible by $4$.
If $J=\varnothing$ and $I\subseteq \{2,\dots,r\}$, then both $n, n/2\in P_I^J$, so that $2|A|+2=np_Ip^J$
and hence $np_Ip^J$ is even and $|A|=\frac{np_Ip^J}{2}-1.$
If $J=\{1\}$, then $n/2\in  P_I^J$ but $n\not\in  P_I^J$, and so $2|A|+1=np_Ip^J$ which
implies that $np_Ip^J$ is odd and $|A|=\lfloor \frac{np_Ip^J}{2}\rfloor.$
In the remaining cases either (1) $J=\varnothing$ and $1\in I$, or (2) $J$ is neither $\varnothing$ nor $\{1\}$.
In both of these cases $n/2\notin P_I^J$; in case (1) we have $n\in  P_I^J$ so that $2|A|+1=|P_I^J|$
and in case (2) we have $n\not\in  P_I^J$ so that $2|A|=|P_I^J|$. Thus in both cases
$|A|=\lfloor \frac{np_Ip^J}{2}\rfloor$ and part ii)a) is proved as well as part iv) with $n\equiv 2\pmod{4}$: note that the condition $n>2$ in iv) guarantees
that $ \frac{np^R}{2}$ is an integer.
The particular situation with $I=\varnothing$ and $|J|=2$ arises in case (2) and here we have
$|A|= \frac{np_Ip^J}{2}$, which proves part iii) for $n\equiv 2\pmod{4}$.

Finally assume that $n$ is divisible by $4.$ If $J=\varnothing$, then both $n, n/2\in P_I^J$ and
$2|A|+2=np_Ip^J;$ so $np_Ip^J$ is even and $|A|=\frac{np_Ip^J}{2}-1.$ If $J\neq \varnothing$,
then neither $n$ nor $n/2$ lies in $P_I^J$ and hence $2|A|=np_Ip^J, $ so that  $np_Ip^J $ is even
and $|A|=\frac{np_Ip^J}{2}.$ Thus part ii)b) is proved as well as parts iii) and iv) for $n$ divisible by $4$.

\end{proof}
\subsection{Maximal subgroups}\label{maximal}

Let $G=S_n, A_n.$ As mentioned in Remark \ref{rem1}~(a), in order to determine $\gamma(G)$ we may assume that the basic components of a
normal covering for $G$ are maximal subgroups of $G.$ These subgroups may be intransitive,
primitive or imprimitive. Each maximal subgroup of $A_n$ which is
intransitive or imprimitive is obtained as $M\cap A_n$, where $M$ is maximal in $S_n$ and respectively intransitive or imprimitive. Thus to give an overview of the maximal subgroups of $G$ which are intransitive or imprimitive, it is enough to describe them for $G=S_n.$

The set of maximal subgroups of $S_n$ which are
intransitive is given, up to conjugacy, by
\begin{equation}\label{P}
\mathcal{P}:=\{\ P_x=S_x\times S_{n-x} \  : 1\leq x< n/2\}.
\end{equation}
If $\mathfrak{X}\subseteq \mathbb{N}\cap [1,n/2)$
we use the notation
\begin{equation}\label{PX}
\mathcal{P}_\mathfrak{X}:=\{\ P_x\in \mathcal{P}\ :\ x\in \mathfrak{X}\}.
\end{equation}
The set of imprimitive maximal subgroups of $S_n$ is given, up to conjugacy, by
$$
\mathcal{W}:=\{\ S_b\wr S_{m} \ : 2\leq b\leq n/2, b\mid n, m=n/b \}.
$$
Recall that, for $n$ even, the intransitive subgroup $P_{n/2}= S_{n/2}\times S_{n/2}$
is not maximal in $S_n$ because it is contained in the imprimitive subgroup $S_{n/2}\wr S_2.$

Intransitive subgroups play a major role in normal coverings of the symmetric and alternating group and we conjecture that, apart from at most two components,
each minimal basic set consists of intransitive components.
To be more precise we make the following conjecture.

\begin{con}\label{conjecture-intr-sa}
Let $n$ as in \eqref{factors}, with $r\geq 2.$ Then, for each  minimal
basic set $\delta$ of $S_n$ consisting of maximal subgroups, the subset of intransitive 
subgroups in $\delta$ is precisely
\begin{equation}\label{pminS}
\mathcal{P}_{\min}(S_n):=\{P_x\in \mathcal{P}\ :  \  \gcd(x,p_1p_2)=1\}.
\end{equation}  
If $n$ is even then, for each  minimal
basic set $\delta$ of $A_n$ consisting of maximal subgroups, the subset of intransitive 
subgroups in $\delta$ is precisely
\begin{equation}\label{pminA}
\mathcal{P}_{\min}(A_n):=\{P_x\cap A_n\ :  P_x \in \mathcal{P}_{\min}(S_n) \}.\end{equation}
 \end{con}

Note that in Conjecture~\ref{conjecture-intr-sa} we do not exclude $n=p_1p_2.$
%
We will see in Proposition \ref{upper} that, when $r\geq 2$, also two imprimitive subgroups play a role in the normal
coverings of $S_n$ and $A_n$.

\subsection{Partitions}
Let $n,k\in\mathbb{N}$, with $k\leq n$. A $k$-{\em partition} of $n$ is an unordered $k$-tuple
$T=[x_1,\dots,x_k]$, with $x_i\in\mathbb{N}$ for each $i\in
\{1,\dots,k\},$ such that $n=\sum_{i=1}^{k}x_i.$
We sometimes simply refer to a $k$-partition as a partition.
The $x_i$ are called the {\em terms} of the $k$-partition. 
Let $\sigma\in S_n$  and let $X_1,\dots, X_k$,
 with $ k\in\mathbb{N},$  be the orbits of $\langle\sigma\rangle.$
Let $x_i=|X_i|\in\mathbb{N}$. Note that the fixed points of $\sigma$ correspond to the $x_i=1,$ while the
lengths of the cycles in which $\sigma$ splits are given by the $x_i\geq 2.$
Then $\sum_{i=1}^kx_i=n$ and we say that $T=[x_1,\dots, x_k]$ is the {\it partition associated} to $\sigma$
or the {\em type} of $\sigma.$ For instance if $\sigma=(123)\in S_4$, then $X_1=\{1,2,3\}, X_2=\{4\}$ and
the type of $\sigma$ is $[1,3].$ Since permutations in $S_n$ are conjugate if and only if they have the
same type, we can identify the conjugacy classes of permutations of $S_n$ with the partitions of $n$ (for more details see Section 1.1 in \cite{BPS}). When a subgroup $H$ of $S_n$ contains a permutation of
type $T$ we say that `$T$ belongs to $H$' and we write $T\in H.$

These concepts are crucial for our research: for a set $\delta$ of subgroups of $S_n$ is a basic set if
and only if, for each partition $T$ of $n,$ there exists $H\in\delta$ such that $T$ belongs to $H.$

In particular the following set of $2$-partitions will be important for our work:

\begin{equation}\label{T}
\mathcal{T}:=\{T_x=[x,n-x] : 1\leq x< n/2\}.
\end{equation}
Note that, when $n$ is even, we exclude the partition $[n/2,n/2].$
The partitions in $\mathcal{T}$ correspond to the simplest types of permutations in $S_n$ apart from $n$-cycles, and allow us to
identify most of the intransitive components in a basic set. The counting problems arising are very easily managed by the following remark.

\begin{remark}\label{bijection}{\rm 
The maps $f:\mathbb{N}\cap [1,n/2)\to \mathcal{T},\  f(x)=T_x$ and $F:\mathcal{T}\to \mathcal{P},\ F(T_x)=P_x$,
with $\mathcal{P}$ and $\mathcal{T}$ as in \eqref{P} and \eqref{T},  are bijections and, for each $x$,  $P_x$ is
the only subgroup in $\mathcal{P}$ which contains a permutation of type $T_x\in \mathcal{T}.$
In particular
 $|\mathcal{T}|=|\mathcal{P}|=\lfloor\frac{n-1}{2}\rfloor$, and for each  $\mathfrak{X}\subseteq \mathbb{N}\cap [1,n/2)$ we have
  $|\mathcal{P}_\mathfrak{X}|=|\mathfrak{X}|,$ with $\mathcal{P}_\mathfrak{X}$ as in (\ref{PX}).
  }
\end{remark}

The subset of $\mathcal{T}$ given by
\begin{equation}\label{Apa}
\mathcal{A}:=\{T_x\in \mathcal{T}\ : x\in \mathfrak{A}\},
\end{equation}
where
\begin{equation}\label{Anu}
\mathfrak{A}:=\{x\in\mathbb{N}: 1\leq x< n/2,\  \gcd(x,n)=1\}
\end{equation}
and  the corresponding set of intransitive components
 $\mathcal{P}_{\mathfrak{A}}$
play an important role in this paper. Note that, by Proposition \ref{arithm} iv) and Remark \ref{bijection}, we have
 \begin{equation}\label{order}
|\mathcal{A}|=|\mathfrak{A}|=| \mathcal{P}_{\mathfrak{A}}|=\frac{\varphi(n)}{2}.
\end{equation}

\section{An upper bound}

We produce a normal covering of $S_n$ with a basic set of size $g(n)$  involving
only intransitive and imprimitive subgroups, for any non-prime-power $n$ (even or odd). 

\begin{proposition}  \label{upper}
Let $n$ be as in \eqref{factors} with $r\geq2.$ Then $\gamma(S_n)\leq g(n)\ \mbox{and}\ \gamma(A_n)\leq g(n).$
\end{proposition}
\begin{proof}
We claim that the set
$$
\delta=\{P_x\in \mathcal{P} \ :\gcd(x,p_1p_2)=1\}\cup\{S_{p_1}\wr S_{n/p_1},\ S_{p_2}\wr S_{n/p_2}\}
$$
does not involve $A_n$ and is a basic set for $S_n$ of size $g(n)$. The argument is somewhat inspired
by the proof of \cite[Proposition 7.8]{BP}. Note that $p_1\geq 2$ and $p_2\geq 3.$
Consider an arbitrary type $T=[x_1,\dots,x_k]$ of permutations in $S_n$, with each $x_i\in\mathbb{N}$,
$k\geq1$, and $\sum_{i=1}^k x_i=n$. If each $x_i$ is divisible by $p_1$ then $T\in S_{p_1}\wr S_{n/p_1}$
in $\delta$, and if each $x_i$ is divisible by $p_2$ then $T\in S_{p_2}\wr S_{n/p_2}$ in $\delta$.
In particular $\delta$ covers the $n$-cycles. So we may assume that $k\geq 2$  and that both
$K_1=\{i\in\{1,\dots,k\}: p_1\nmid x_i\}$ and $K_2=\{i\in\{1,\dots,k\}: p_2\nmid x_i\}$
are nonempty subsets of $K=\{1,\dots,k\}.$

If $K_1\cap K_2\neq\varnothing$, then there exists $i\in K$ such that  $p_1,p_2\nmid x_i.$
In particular we have $x_i\neq n/2$, as otherwise $p_2,$ which divides $n$, would divide $x_i$.
Let $x$ be the unique natural number in $\{x_i,n-x_i\}$ which is less than $n/2.$ Note that,
since $p_1,p_2\mid n,$ we have $p_1,p_2\nmid x.$ Thus
$T\in P_x$  with $P_x\in \delta.$

Assume next that $K_1\cap K_2=\varnothing:$ then
for all $i\in K$ we have that if $p_1\nmid x_i$ then $p_2\mid x_i$ and that if $p_2\nmid x_i$
then $p_1\mid x_i.$
Let $i\in K_1$ and $j\in K_2$. Then we have $i\neq j$, and $p_1\nmid x_i,\ p_2\mid x_i$,
and also $p_2\nmid x_j,\ p_1\mid x_j.$ This implies that $p_1\nmid x_i+x_j$
and $p_2\nmid x_i+x_j$, and it follows that $x_i+x_j\neq n, n/2$.
Let $x$ be the unique natural number in $\{x_i+x_j,n-(x_i+x_j)\}$ which is less than $n/2:$ then $p_1,p_2\nmid x$ and so $T\in P_x$ with $P_x\in \delta.$
Now, to get the result for $A_n$, use Remark \ref{rem1}b).
\end{proof}
 We recall that for $n=p_1p_2,$ the inequalities in Proposition~\ref{upper} are  strict since,  by \cite[Proposition 7.6]{BP},  for $n$ odd we have $\gamma(S_{p_1p_2})= g(p_1p_2)-1$ and for $n$ even we have $\gamma(A_{p_1p_2})= g(p_1p_2)-1.$
Note that the basic set
in the proof of Proposition \ref{upper} admits as intransitive components precisely those belonging
to the set $\{P_x\in \mathcal{P}\ :  \  \gcd(x,p_1p_2)=1\}$, in line with Conjecture~\ref{conjecture-intr-sa}.
We can also observe that Conjecture~\ref{conjecture-intr-sa} is stronger of both Conjecture~\ref{conjecture-sym} and Conjecture~\ref{conjecture-alt}
 and that it implies a characterization of the minimal basic sets for $S_n$ when $n$ is odd, and for $A_n$ when $n$ is even. Recall the definitions of
  $\mathcal{P}_{\min}(S_n)$ and $\mathcal{P}_{\min}(A_n)$ given in (\ref{pminS}) and (\ref{pminA}).

\begin{proposition}\label{conjectures} Assume that Conjecture~\ref{conjecture-intr-sa} holds. Let $$\mathcal{W}_i:= \{S_{p_i}\wr S_{n/p_i}, S_{n/p_i}\wr S_{p_i} \}$$ for $i=1,2.$ Then:
\begin{itemize}
\item[i)] Conjecture~$\ref{conjecture-sym}$ and Conjecture~$\ref{conjecture-alt}$ hold.

\item[ii)] If $n$ is odd, with $r\geq 2$ and $n\neq p_1p_2$, then the only minimal basic sets of $S_n$ consisting of maximal subgroups are $\delta=\mathcal{P}_{\min}(S_n)\cup \{I_{p_1},I_{p_2}\},$ where $I_{p_i}\in \mathcal{W}_i,$  for $i=1,2.$

\item[iii)] If $n$ is even, with $r\geq 2$ and $n\neq 2p_2$, then the only minimal basic sets of $A_n$ consisting of maximal subgroups are $\delta=\mathcal{P}_{\min}(A_n)\cup \{I_{p_1}\cap A_n,I_{p_2}\cap A_n\},$ where $I_{p_i}\in  \mathcal{W}_i,$ for $i=1,2.$

\end{itemize}

\end{proposition}
\begin{proof} Let $n\in\mathbb{N}$ with $r\geq 2$ and $n\neq p_1p_2.$ Let
$\delta$ be a minimal basic set for $G$ consisting of maximal subgroups. Since Conjecture~\ref{conjecture-intr-sa} holds by assumption, 
the intransitive components in $\delta$ are precisely the $g(n)-2$ subgroups in $\mathcal{P}_{\min}(G)$. Moreover, by Proposition \ref{upper}, we know that $|\delta|\leq g(n)$.
Consider the types $T_i=[p_i,n-p_i]$ for $i=1,2.$ Observe that, by assumption, we have $n>p_1p_2$ and therefore also $p_i<n-p_i$ for $i=1,2.$
Note that the types $T_i$ belong to $A_n$ if and only if $n$ is even, and that
 they do not belong to any subgroup in $\mathcal{P}_{\min}(G).$  Let first $G=S_n,$ with $n$ odd or $G=A_n,$ with $n$ even.
 Thus there exists a transitive subgroup $H_1$ maximal in $G$, with $H_1\in\delta$ and $H_1$ containing $T_1$. Note that $H_1\neq A_n$ because $G=S_n$ is considered only for $n$ odd.
 Assume that $H_1$ is primitive and examine the list of the primitive subgroups  of $S_n$ containing a permutation of type $[k,n-k]$ with $k=p_i<n/2,$
 in Theorem 3.3 of \cite{M}.
Since $n$ is not a proper power, $n\neq p_1p_2$ and $k$ is a  prime, there is no such primitive subgroup $H_1$.  
It follows that $H_1$ is imprimitive and the only possibility is $H_1=I_{p_1}\cap G$, for some $I_{p_1}\in \mathcal{W}_1.$ Since $T_2=[p_2,n-p_2]$ does not belong to any subgroup in $\mathcal{W}_1,$ the same argument shows that $\delta$ contains also a component $H_2=I_{p_2}\cap G$, for some $I_{p_2}\in \mathcal{W}_2.$ In particular $\gamma(G)=g(n).$
 So we have proved ii) and iii) as well as i) except for $S_n$ with $n$ even. In this 
 last case we consider separately $A_n\notin\delta$ and $A_n\in\delta.$ If $A_n\notin\delta$ then, as above, we need two further components to cover $T_1$ and $T_2$, and we find $\gamma(S_n)=g(n).$ If $A_n\in\delta$ we consider the type $T=[2,2,n-4]$: since $T$ does not belong to $A_n$ or to any subgroup in $\mathcal{P}_{\min}(S_n),$ we conclude that $\delta$ has a further component containing $T$ and thus, again,$\gamma(S_n)=g(n).$

\end{proof}

In other words if Conjecture \ref{conjecture-intr-sa} is true, we have a complete description of those minimal normal coverings consisting of maximal subgroups, both for $S_n$ ($n$ odd) and $A_n$ ($n$ even), with $n$ not a prime power or the product of two primes. 
In particular, in these cases, this rules out any role for primitive
  subgroups in normal coverings.

\section{Imprimitive subgroups containing permutations with at most four cycles of globally coprime lengths}

Let $n\in\mathbb{N}$ be composite, say $n=bm$ where $b\in \mathbb{N}$ with $2\leq b\leq n/2$ and $m=n/b.$ Let $W=S_b\wr S_{m}\in\mathcal{W}$,
a maximal imprimitive subgroup of $S_n$ stabilising a block system $\mathcal{B}$ consisting of
$m$ blocks of size $b$.
Let $\sigma\in S_n$  such that the partition $T=[x_1,\dots, x_k]$ associated with $\sigma$
satisfies $k\leq 4$ and $\gcd(x_1,\dots, x_k)=1$, that is to say,
the $x_i$ are `globally coprime'. Note that, if $k=2$, then these partitions are the $2$-partitions in 
the set $\mathcal{A}$ defined in (\ref{Apa}).
  Let $X_1,\dots, X_k$ be the
corresponding $\langle\sigma\rangle$-orbits, and let $\sigma_i=\sigma|_{X_i}$ for each $i$.
In this section we describe how to check if  $W$  contains a
conjugate of $\sigma,$ that is, a permutation of type $T$, echoing some ideas in the proof of
\cite[Lemma 4.2]{BPS}. 

For $i\in\{1,\ldots,k\},$ let $\mathcal{B}_i=\{B\in \mathcal{B}\mid B\cap
X_i\neq \varnothing\}$. Observe that  $\langle\sigma_i\rangle$ acts
transitively on $\mathcal{B}_i$ and that the action of $\sigma_i$ on
$\mathcal{B}_i$ is equivalent to the action of $\sigma$ on
$\mathcal{B}_i$. It follows that $d_i=|B\cap X_i|$ is independent of the choice of
$B\in\mathcal{B}_i:$ in particular $d_i\mid x_i=|X_i|$ and, of course, $1\leq d_i\leq b.$
Moreover, since the orbits $X_i$ form a partition of $\Omega=\{1,\dots,n\},$ we also have $\mathcal{B}_i\cap
\mathcal{B}_j=\varnothing$ or $\mathcal{B}_i=\mathcal{B}_j$, for
distinct $i,j\in \{1,\ldots,k\}$, and $\cup_{i=1}^k \mathcal{B}_i=\mathcal{B}.$

In addition, the assumption $\gcd(x_1,\ldots,x_k)=1$ leads to more restrictions on $\mathcal{B}.$
For example,  we cannot have $\mathcal{B}_i=\mathcal{B}$ for some $i\in
\{1,\ldots,k\}$, since this would imply that $\mathcal{B}=\mathcal{B}_i$ for all $i$ and hence that
$|\mathcal{B}|=m$ divides $\gcd(x_1,\ldots,x_k)$. Similarly we cannot have
$\mathcal{B}_i\cap \mathcal{B}_j=\varnothing$ for all distinct $i, j\in\{1,\ldots,k\}$,
since this would imply that $|B|=b$ divides $\gcd(x_1,\ldots,x_k)$.

\subsection{The case k=2}\label{case k=2}
As a first consequence of the comments above, we obtain that 
there exists no maximal imprimitive subgroup of $S_n$ containing a permutation of type
$T=[x,n-x]\in\mathcal{A}$: for we showed above that we must have either
$\mathcal{B}_1\cap \mathcal{B}_2 =\varnothing $ or $\mathcal{B}_1=\mathcal{B}_2=\mathcal{B},$
and the previous paragraph argues that neither of these is possible.

\subsection{The case k=3}\label{k=3}
Let now $k=3$, that is, $T=[x_1,x_2,x_3]$ with $\gcd(x_1,x_2,x_3)=1$. From the
previous paragraphs, up to relabeling,  we may assume that
$\mathcal{B}_1=\mathcal{B}_2$ and that $\mathcal{B}_3\cap
\mathcal{B}_1=\varnothing$.
Thus there is just one pattern to examine:
\begin{equation}\label{3a}
\begin{array}{|c|c|}
\hline
       &    x_1\\ \cline{2-2}
x_3 &  \\
       &  x_2\\\hline
\end{array}
\end{equation}

Here $X_3$ is a union of blocks while $X_1, X_2$ share the same set of blocks.
In particular $b=d_3=d_1+d_2$ divides $\gcd(x_3,n).$ The number of blocks in $X_3$ is  $x_3/b$
and the remaining $m-x_3/b$ blocks meet both $X_1$ and $X_2$ so that $\frac{n-x_3}{b}\mid\gcd(x_1,x_2).$

The two conditions $b\mid \gcd(x_3,n)$ and $\frac{n-x_3}{b}\mid\gcd(x_1,x_2)$ are quite strong
and can often be used to prove that $T$ does not belong to $W$.
The following examples play a role in the proof of Theorem \ref{thm1}.
\begin{example}\label{Z} No permutation of  type $Z=[3, q,14q-3]$ or $X=[10,4q,11q-10]$
belongs to an imprimitive subgroup of $S_{15q}$, for $q\geq 7$ prime.
\end{example}
\begin{proof} Clearly the terms in $Z$ are pairwise coprime. In particular they are globally coprime.
Due to the relabeling we made in our argument above, we must examine the unique possible pattern (\ref{3a}) for each of the three choices of $x_3$ in $\{3, q,14q-3\}.$ For each choice,
we get $\frac{15q-x_3}{b}\mid\gcd(x_1,x_2)=1$ so that $b=15q-x_3$ must divide $n$,
which is not possible for any of the choices for $x_3$.

Now consider the partition $X$. 
If $x_3=10,4q$, then $b \mid \gcd(x_3,15q)$ implies that $b=5,q$, against the fact that $\frac{n-x_3}{b}=3q-2, 11$
does not divide $\gcd(x_1,x_2)=1,10$, respectively. Hence $x_3=11q-10$ which implies that
 $\frac{n-x_3}{b}=\frac{4q+10}{b}$ divides $\gcd(x_1,x_2)=2$; so $b=2q+5$ or $4q+10$,
neither of which divides $n=15q$.
%
 \end{proof}

 \subsection{The case k=4}
Let us explore finally $k=4$, that is  $T=[x_1,x_2,x_3,x_4]$ with  $\gcd(x_1,\ldots,x_4)=1$. From the
previous paragraphs, up to relabeling,  we may assume that
$\mathcal{B}_1=\mathcal{B}_2$ and that $\mathcal{B}_3\cap
\mathcal{B}_1=\mathcal{B}_3\cap\mathcal{B}_2=\varnothing$.
Thus there are three cases (a), (b), (c) to examine, each characterized by some
additional conditions and giving rise to a different pattern.
\begin{itemize}
\item[(a)] $\mathcal{B}_3\cap\mathcal{B}_4=\varnothing, \mathcal{B}_1=\mathcal{B}_4.$
\begin{equation}\label{4a}
\begin{array}{|c|c|}
\hline
       &    x_1\\ \cline{2-2}
x_3 &  \\
       &  x_2\\ \cline{2-2}
       & x_4\\\hline
\end{array}
\end{equation}
Here $X_3$ is a union of blocks while $X_1,X_2,X_4$ share the same blocks. In particular
$b=d_3=d_1+d_2+d_4\mid \gcd(x_3,n).$ The number of blocks in $X_3$ is $x_3/b$ and the remaining
$m-x_3/b$ blocks meet each of $X_1,X_2, X_4$ so that $|\mathcal{B}_1|=\frac{n-x_3}{b}\mid\gcd(x_1,x_2, x_4).$
\medskip

\item[(b)] $\mathcal{B}_3\cap\mathcal{B}_4=\varnothing, \mathcal{B}_4\cap\mathcal{B}_1=\varnothing,
\mathcal{B}_4\cap\mathcal{B}_2=\varnothing.$

\begin{equation}\label{4b}
\begin{array}{|c|c|c|}
\hline
       & &    x_1\\ \cline{3-3}
x_3 &x_4  &  \\
       & &  x_2\\\hline
\end{array}
\end{equation}
Here $X_3, X_4$ are disjoint unions of blocks while $X_1,X_2$ share the same blocks.
In particular $b=d_3=d_4=d_1+d_2\mid \gcd(x_3,x_4,n).$ The number of blocks contained
in  $X_3\cup X_4$ is given by $(x_3+x_4)/b$ and the remaining $m-(x_3+x_4)/b$ blocks
meet both $X_1,X_2.$ Therefore $\frac{n-x_3-x_4}{b}\mid\gcd(x_1,x_2).$
\medskip
\item[(c)] $\mathcal{B}_3=\mathcal{B}_4, \mathcal{B}_4\cap\mathcal{B}_1=\varnothing, \mathcal{B}_4\cap\mathcal{B}_2=\varnothing.$ \begin{equation}\label{4c}
\begin{array}{|c|c|}
\hline
  x_3     &    x_1\\ \cline{2-2}
   & \\ \cline{1-1}
x_4 &  \\
       &  x_2\\\hline
\end{array}
\end{equation}
Here $X_3, X_4$ share the same blocks and $X_1,X_2$ share the same remaining blocks.
In particular $b=d_3+d_4=d_1+d_2\mid \gcd(x_3+x_4,n)$. The number of blocks contained
in $X_3\cup X_4$ is  $|\mathcal{B}_3|=\frac{x_3+x_4}{b}$ and divides $\gcd(x_3,x_4)$. Moreover
 $\frac{n-x_3-x_4}{b}$ divides $\gcd(x_1,x_2)$, the quotient giving the number of blocks contained in $X_1\cup X_2.$

\end{itemize}
The arithmetic conditions deduced in these cases are also often sufficiently strong to prove that $T$
does not belong to $W$. 
Anyway, if necessary, we can strengthen these conditions excluding some congruences.
The two following examples  play a role in the proof of Theorem \ref{thm1}.

\begin{example}\label{U,V}
Let $q\geq 7$ be a prime. Then:
\begin{itemize}\item[i)] no permutation of  type $U=[5,q-5,10q+5, 4q-5]$ belongs to an imprimitive
subgroup of $S_{15q};$
\item[ii)] if  $q\geq 11$ and $q\not\equiv 12 \pmod{13}$, then no permutation of  type $$V=[q-7,q+7,6q-7, 7q+7]$$ belongs
to an imprimitive subgroup of $S_{15q}.$
\end{itemize}
\end{example} 
\begin{proof}

To begin with note that for both $U$ and $V$ the greatest common divisor of their terms is $1$
so that these partitions are within the category we are exploring in this section.
Let $W=S_b\wr S_{m}$, with $2\leq b\leq n/2$ a proper divisor of $n$ and $m=n/b.$
Thus $b, m\in A=\{3,5,15,q,3q,5q\}.$

i) Assume first  that $U$ belongs to $W$. We show that each of the three possible  cases (a),(b),(c), corresponding respectively to the patterns
(\ref{4a}),(\ref{4b}),(\ref{4c}), leads to a contradiction.

 Case (a).\quad We must examine each of the possibilities for
$$
x_3\in\{5,q-5,10q+5, 4q-5\}.
$$
Recall that in this context we have $b\mid \gcd(x_3,n)$ and $\frac{15q-x_3}{b}\mid \gcd(x_1,x_2,x_4).$
For each choice of $x_3$ we have $\gcd(x_1,x_2,x_4)=1$, and so the second condition implies that
$b=15q-x_3$. However, by the first condition $b\leq x_3$, and it follows that $2x_3\geq 15q$,
whence $x_3=10q+5$. Thus $b=5q-5$, which does not divide $15q$.

\smallskip
 Case (b).\quad We must examine  each possibility for $(x_3,x_4)$ in
$$
\{(5,q-5),(5,4q-5),(5,10q+5),(q-5,4q-5),(q-5,10q+5), (4q-5,10q+5)\}.
$$
Recall that now we have  $b\mid \gcd(x_3, x_4,15q)$ and $\frac{15q-x_3-x_4}{b}\mid\gcd(x_1,x_2).$

The first condition implies that $\gcd(x_3,x_4)>1$ and is divisible by a prime $p\in\{3,5,q\}$. This implies that
$(x_3,x_4,p)=(5,10q+5,5)$ or $(q-5,4q-5,3)$, and in these cases we find that $b=5$ or $3$ respectively.
The second condition then requires that $q-2$ divides $\gcd(q-5,4q-5)=\gcd(q-5,3)$ or $\frac{10(q+1)}{3}$
divides $\gcd(5,10q+5)=5$, respectively, neither of which is possible.
%

\smallskip
Case (c).\quad Due to the symmetric role played by $\mathcal{B}_3, \mathcal{B}_4$ with
respect to $\mathcal{B}_1, \mathcal{B}_2,$
we need  to examine only the possibilities for
$$
(x_3,x_4)\in\{(5,q-5),(5,4q-5),(5,10q+5)\}.
$$
Recall that here we have $b\mid \gcd(x_3+x_4,n)$ and the number of blocks contained in $X_3\cup X_4$ divides $\gcd(x_3,x_4);$ moreover
$\frac{n-x_3-x_4}{b}\mid \gcd(x_1,x_2).$

For the first two possibilities for $(x_3,x_4)$, the condition $b\mid \gcd(x_3+x_4,n)$ implies that $b=q$.
In the first case the second condition gives $14\mid 4q-5,$ which is impossible since $4q-5$ is odd.
If $(x_3,x_4)=(5,4q-5)$, then $\gcd(5,4q-5)=1$, which implies that $|\mathcal{B}_3|=1$ and so $b=x_3+x_4=4q$,
a contradiction.
Thus $(x_3,x_4)=(5,10q+5)$, and we find that $|\mathcal{B}_3|$ divides $\gcd(5,10q+5)=5$ and cannot be $1$,
as otherwise $b=x_3+x_4=10(q+1)\notin A.$
Hence $|\mathcal{B}_3|=5$ and $b=2(q+1)\notin A.$
\bigskip

ii) The proof follows with a case-by-case argument similar to the proof of part i). The peculiarity  with respect to i) is in case
(c) where, to conclude, we need  the assumption $q\not\equiv 12 \pmod{13}.$ We treat in detail this case, having in mind the corresponding pattern (\ref{4c}).
\smallskip

Case (c).\quad Due to the symmetric role played by $\mathcal{B}_3, \mathcal{B}_4$ with
respect to $\mathcal{B}_1, \mathcal{B}_2,$
we need to examine only the possibilities for
$$
(x_3,x_4)\in\{(q-7,q+7),(q-7,6q-7),(q-7,7q+7)\}.
$$
Here we have $b\mid \gcd(x_3+x_4,n)$ and  
$\frac{n-x_3-x_4}{b}\mid \gcd(x_1,x_2).$
In the first and third cases, the condition $b\mid \gcd(x_3+x_4,n)$ implies that $b=q$. The second condition then implies
that  $13$ divides $\gcd(6q-7,7q+7)=\gcd(q+1,13)$, or that $7$ divides $\gcd(q+7,6q-7)= 1,$ respectively,
neither of which is possible, because of our assumption  $q\not\equiv 12 \pmod{13}$.
Thus $(x_3,x_4)=(q-7,6q-7)$, and the first condition implies that  $b\mid \gcd(7q-14,15q)=\gcd(q-2,15)$
and so $b\in\{3,5,15\}.$ On the other hand we have $\gcd(q-7,6q-7)\mid 5$ and thus $|\mathcal{B}_3|\in\{1,5\};$
if $|\mathcal{B}_3|=1$ then $b=7q\notin A$ and thus $|\mathcal{B}_3|=5, b=\frac{7q-14}{5},$
which is incompatible with $b\in\{3,5,15\}.$
%
%
\end{proof}

\section{Products of  three odd primes}\label{main}

In this section we deal with degrees of the form $n=15q,$ where $q\geq 7$ is a prime. We will see that Conjecture \ref{conjecture-sym} holds for
an infinite family of these degrees.
We begin this section adapting to our scope some deep classical and recent results about primitive groups.

\begin{lemma}\label{Muller}
Let  $q\geq 7$ be a prime. Then no primitive proper subgroup of $S_{15q}$ contains a permutation
with type belonging to $\mathcal{T}$.
\end{lemma}

\begin{proof}  We simply examine the lists in  \cite[Theorem 3.3]{M}  having in mind that  $n$ is odd
and not a proper power, and that no type in $\mathcal{T}$ can belong to $A_{15q}.$
It is easily checked that no case arises.
\end{proof}

\begin{lemma}\label{GM} Let  $q\geq 7$ be a prime and let $K$ be a primitive subgroup of $S_{15q},$ with
$K\ngeq A_{15q}.$ Then the number of fixed points of each nontrivial permutation in $K$ is less than $9q.$
\end{lemma}

\begin{proof}  By \cite[Corollary 1]{GM}, since $n=15q$ is not a proper power, we find that the number
of fixed points of each permutation in $K$ is less than or equal to $\frac{4}{7} (15q),$ and thus, in
particular, is less than $9q.$
\end{proof}

\begin{lemma} \label{extract} {\em (~\cite[Theorem 13.8]{WI},  \cite[Theorem 4.11]{CA} )}
A primitive group of degree $n$, which contains a permutation of type $[m,1,\dots,1]$ where
$2\leq m\leq n-5,$ contains $A_n$.
\end{lemma}


\bigskip\noindent
\textit{Proof of Theorem~\ref{thm1}~(a)}\quad
Let $n=15q$, with $q$ an odd prime satisfying  $q\equiv 2 \pmod{15}$ and $q\not\equiv 12 \pmod{13}$, so that $q\geq 17.$
Let $\delta$ be  a minimal basic set for $S_{n}.$ We may assume that each component in $\delta$
is a maximal subgroup of  $S_{n}.$ By Proposition \ref{upper}, $|\delta|\leq4q+2$. We show that also $|\delta|\geq 4q+2.$

We are interested in the set $\mathcal{A}$ defined in (\ref{Apa}) and also in the
following subsets  of $\mathcal{T},$ corresponding to subsets of $\mathfrak{T}:=\mathbb{N}\cap [1,n/2)$ characterized by a suitable coprimality condition. \[
\begin{array}{lll} 
\mathcal{B}=\{T_x\in \mathcal{T}\  : \ x\in \mathfrak{B}\},& & \mathfrak{B}=\{x\in \mathfrak{T}: \ \gcd(x,n)=3\},\\
\mathcal{C}=\{T_x\in \mathcal{T}\  : \ x\in\mathfrak{C} \},& & \mathfrak{C}=\{x\in\mathfrak{T}\ : \ \gcd(x,n)=5\},\\
\mathcal{D}=\{T_x\in \mathcal{T}\  : \ x\in \mathfrak{D} \},& & \mathfrak{D}=\{x\in\mathfrak{T}\ : \ \gcd(x,n)=q\}.\\
\end{array}\]
Their orders are immediate and coincide with that of the corresponding set of intransitive maximal subgroups containing them.
By  equality (\ref{order}), Proposition \ref{arithm}~i) and Remark \ref{bijection} we have:
\[\begin{array}{lll}| \mathcal{A}|=| \mathfrak{A}|=|\mathcal{P}_\mathfrak{A}|=4(q-1),& & | \mathcal{B}|=| \mathfrak{B}|=|\mathcal{P}_\mathfrak{B}|=2(q-1),\\
& &\\
 | \mathcal{C}|=| \mathfrak{C}|=|\mathcal{P}_\mathfrak{C}|=q-1,& & | \mathcal{D}|=| \mathfrak{D}|=|\mathcal{P}_\mathfrak{D}|=4.
\end{array}\]

Note that since $n$ is odd, no permutation of type belonging to  $\mathcal{T}$ lies in $A_n.$ Thus
by Lemma \ref{Muller},  no permutation having type in $\mathcal{T}$ belongs to a primitive subgroup in $\delta$.
In particular this holds for all the types in $\mathcal{A}, \mathcal{B},\mathcal{C},\mathcal{D}.$
By  Section \ref{case k=2}, it follows that the permutations having type in $\mathcal{A}$ belong only to intransitive components and so
$\delta\supseteq \mathcal{P}_\mathfrak{A},$ which gives $|\delta|\geq 4(q-1).$ This means that we need to
force only $6$ further components to finish the proof.

Consider now the permutations having type belonging to $\mathcal{B}$. Since these do not belong to any of
the subgroups in $\mathcal{P}_\mathfrak{A},$ we need other components in $\delta$ to contain them and
those components must be intransitive or imprimitive. If both $S_3\wr S_{5q}\notin \delta$ and
$S_{5q}\wr S_{3}\notin \delta,$ then we would need a further $2(q-1)>6$ intransitive components and thus
$|\delta|>4q+2$, a contradiction. So one of
$S_3\wr S_{5q}$ and $S_{5q}\wr S_{3},$ belongs to $\delta$. Let $I_3$ denote this component  of $\delta$,
and note that, since now $|\delta|\geq 4q-3,$ we need to force only $5$ further components.

As $|\mathcal{C}|=q-1>6, $ looking to the permutations having type belonging to $\mathcal{C}$, and observing that they do not
belong  to any of the subgroups in $\mathcal{P}_\mathfrak{A}$  or to $I_3$, we see that $\delta$ must contain $S_5\wr S_{3q}$ or  $S_{3q}\wr S_{5}.$
 Let  denote with $I_5$  this
additional subgroup in $\delta.$ 

At this point we know that $\delta\supseteq \mathcal{P}_\mathfrak{A}\cup \{I_3,I_5\}$ and
we need to force just $4$ components.
If neither $S_q\wr S_{15}$ nor $S_{15}\wr S_{q},$ belongs to $\delta,$ to cover the permutations
of type belonging to $\mathcal{D}$ we need exactly $4$ additional intransitive components and the proof is complete.
So we may assume that  $I_q\in \delta$, where $I_q$ is one of $S_q\wr S_{15}$ or $S_{15}\wr S_{q}.$

Next suppose that $A_{15q}\in\delta$. We consider the type $U=[5,q-5,10q+5, 4q-5].$ By Example \ref{U,V}
no imprimitive subgroup contains a permutation of type $U.$ Moreover $U\notin P_x$ for all $x\in \mathfrak{A}$
because no term and no sum of two terms in the partition $U$ is coprime to $n=15q$: this follows from
the assumption $q\equiv 2 \pmod{15}$, which implies both $q\equiv 2 \pmod{3}$ and $q\equiv 2 \pmod{5}$
so that $3\mid q-5, 4q-5,q+7, 7q+7$ and $5\mid q-7, 6q-7$. Also $U\notin A_{15q}.$ This means that we need
a further component, say $K$, to cover $U$, and $K$ is either intransitive or primitive.
We now have that $\delta$ contains the subset
$$
\delta':=\{P_x, I_3,I_5, I_q, A_{15q}, K \ :\  x\in\mathfrak{A}\}
$$
of size $4q+1$. Suppose, for a contradiction, that $\delta=\delta'$, and let $\sigma\in K$ have type $U$.
Suppose first that $K$ is primitive: then $\mu=\sigma^{10q+5}\neq id,$ because $3\mid q-5$, due to
$q\equiv 2 \pmod{3}$,  but $3\nmid 10q+5.$ Moreover the number of fixed points of $\mu$ is at least
$10q+10>9q,$ contradicting Lemma \ref{GM}. Thus $K$ is intransitive. To be more precise
$$
K\in P_{U} :=\{P_{5},P_{q},P_{q-5},P_{4q-5}, P_{5q-10},P_{4q},P_{5q-5}\}.
$$
Consider now the type $V=[q-7,q+7,6q-7,7q+7].$ By Example \ref{U,V} no imprimitive subgroup contains
a permutation of type $V$ and, arguing as for $U$,  it is immediately checked that $V \notin P_x$ for all
$x\in \mathfrak{A}.$ On the other hand for each choice of $K$ in  $P_{U}$, we see that $V\notin K.$
In other words it is not possible to cover $V$ by the components in $\delta.$ Thus if $\delta$ contains $A_{15q}$, then
$|\delta|=4q+2$.

Finally suppose that  $A_{15q}\notin\delta.$ We consider  the types $Z=[3,q,14q-3]$ and $X=[10,4q,11q-10].$
Since each term and the sum of each pair of terms in $Z,X$ is divisible by $3$ or by $5$ or by $q$, we have
that both $Z$ and $X$ do not belong to any intransitive component of $\mathcal{P}_{\mathfrak{A}}.$
On the other hand, by Example \ref{Z} they do not belong to any imprimitive subgroup. Suppose that $\delta$
contains a primitive component $H< S_{15q}$ containing either $Z$ or $X$.
If $H$ contains an element $\eta$ of type $Z$, then $\eta^{q(14q-3)}\in H$ is a $3$-cycle, which is impossible by Lemma \ref{extract}, because
$H\neq A_{15q}$.
Thus $H$ contains an element $\theta$ of type $X$; but then $\theta^{40q}$ is a $(11q-10)$-cycle and,
since $H\neq A_{15q},$ this is impossible again by Lemma \ref{extract}.
 Hence $\delta$ contains intransitive maximal subgroups containing $X, Z$. Now $Z,X$ belong, respectively,
only to the intransitive subgroups in $\mathcal{P}_Z$ and in $\mathcal{P}_X,$ where
$$
\mathcal{P}_Z:=\{P\in \mathcal{P}: Z\in P\}=\{P_{3},P_{q},P_{q+3}\}
$$
and
$$
\mathcal{P}_X:=\{P\in \mathcal{P}: X\in P\}=\{P_{10},P_{4q},P_{4q+10}\}.
$$

Since  $\mathcal{P}_Z\cap \mathcal{P}_X=\varnothing,$ we need two additional intransitive components
$K_1,K_2$ in $\delta$ to cover both $Z$ and $X$, where  $K_1\in \mathcal{P}_Z, K_2\in \mathcal{P}_X.$
We now have that $\delta$ contains the subset
$$
\delta':= \{P_x, I_3,I_5, I_q, K_1, K_2 \ :\  x\in\mathfrak{A}\}
$$
of size $4q+1$. Assume, for a contradiction, that $\delta=\delta'$.  To reach a final contradiction,
consider again the type $U$ and recall that $U\notin P$ for all $P\in \mathcal{P}_\mathfrak{A}\cup
\{I_3,I_5,I_q\}$, and it is also immediately checked that $U\notin P$ for all
$P\in \mathcal{P}_Z\cup \mathcal{P}_X.$

\hspace{12cm}$\square$

\medskip

We show now that there are infinitely many primes $q$ satisfying the conditions of Theorem~\ref{thm1},
thus confirming Conjecture \ref{conjecture-sym} for a new infinite family of odd integers.

\begin{corollary}\label{support} There exist infinitely many primes $q$ such that $\gamma(S_{15q})=4q+2=g(15q).$

\end{corollary}
\begin{proof} It is enough to observe that, by the previous result,  we have $\gamma(S_{15q})=4q+2$
for all $q$ primes with $q\equiv 2 \pmod{195},$ where $195=15\cdot 13.$
Since $2$ and $195$ are coprime, the famous Theorem of Dirichlet on primes in arithmetic progressions assures that
there are infinitely many primes of the form $q=2+195\,k,$ with
$k\in\mathbb{N}$.
\end{proof}

\section{Even degrees}

In this section we discuss Conjecture \ref{conjecture-alt}. First we justify the exclusion of $n=12$.

\begin{proposition}\label{12}
$\gamma(A_{12})=3$, and $\gamma(S_{12})=g(12)=4.$
\end{proposition}

\begin{proof}
We see that $\delta=\{M_{12}, [S_3\wr S_4]\cap A_{12}, [S_5\times S_7]\cap A_{12}\}$ is a basic set for $A_{12}$  by checking that all the types of permutations in
$A_{12}$ are contained in one of the components. Since $\gamma(A_{12})\neq 2$, by \cite{BBH,BU}, we deduce that $\gamma(A_{12})=3.$

For $S_{12}$, it follows from \cite{BBH} that $\gamma(S_{12})\geq3$, and from Proposition~\ref{upper} that $\gamma(S_{12})\leq g(12)=4$. Suppose for a contradiction that $\delta$ is a basic set for $S_{12}$ consisting of maximal subgroups and that $|\delta|=3$. The proof of \cite[Corollary 7.10]{BP} correctly shows that $A_{12}\not\in\delta$, and then by \cite[Lemma 5.2]{BP} we deduce that $S_5\times S_7\in\delta$.
Since no permutation of type $[2,10], [4,8]$, or $[3,9]$ belongs to a proper primitive 
subgroup, we conclude that neither of the other two subgroups $H,K$ of $\delta$ is 
primitive, and neither $H$ nor $K$ fixes a point. This implies that no subgroup in $\delta$ contains the type $[1,11]$. 
\end{proof}

Now we prove Theorem~\ref{thm1}~(b), giving extra confirmation for  Conjecture \ref{conjecture-alt}.

\bigskip\noindent
\textit{Proof of Theorem~\ref{thm1}~(b).}\quad
Let $n=6q$, with $q\geq 11$ a prime. If $U\leq S_n$, we write for simplicity,
$\overline{U}=U\cap A_n.$ For each $\mathfrak{X}\subseteq \mathbb{N}\cap [1,n/2),$  use the
notation $\overline{\mathcal{P}}_{\mathfrak{X}}=\{\overline{P}: P\in \mathcal{P}_{\mathfrak{X}}\}.$
Let $\delta$ be  a minimal basic set for $A_{n}$ with maximal components. By Proposition \ref{upper} we
know that $|\delta|\leq q+2.$ Our aim is to show that  $|\delta|\geq q+2.$

To do that we first consider the set $\mathfrak{A}'=\mathfrak{A}\setminus \{1\}$ (see (\ref{Anu})) and the
corresponding set of partitions $\mathcal{A}'=\mathcal{A}\setminus \{[1,n-1]\}$ and intransitive
components $\overline{\mathcal{P}}_{\mathfrak{A}'}=\overline{\mathcal{P}}_{\mathfrak{A}}\setminus
\{\overline{P}_1\}$. Moreover, we consider the following subsets  of $\mathcal{T},$ corresponding to
subsets of $\mathfrak{T}:=\mathbb{N}\cap [1,n/2)$ characterized by a suitable coprimality condition:
  \[
\begin{array}{lll}
\mathcal{E}=\{T_x\in \mathcal{T} : x\in \mathfrak{E}\},& &\mathfrak{E}=\{x\in \mathfrak{T}\ : \ \gcd(x,n)=2\} ,  \\
\mathcal{F}=\{T_x\in \mathcal{T} : x\in \mathfrak{F}\},& &\mathfrak{F}=\{x\in \mathfrak{T}\ : \ \gcd(x,n)=3\},   \\
\mathcal{G}=\{T_x\in \mathcal{T}: x\in \mathfrak{G}\},&&\mathfrak{G}=\{x\in \mathfrak{T}\ : \ \gcd(x,n)=q\}.\\
\end{array}\]
It follows from
Proposition \ref{arithm} ii)a) and Remark \ref{bijection} that:
\[
\begin{array}{lll}  | \mathcal{A}'|=|\overline{\mathcal{P}}_{\mathfrak{A}'}|=q-2, & &|\mathcal{E}|=| \mathfrak{E}|=|\overline{\mathcal{P}}_\mathfrak{E}|=q-1,\\
 | \mathcal{F}|=| \mathfrak{F}|=|\overline{\mathcal{P}}_\mathfrak{F}|=\frac{q-1}{2},& & | \mathcal{G}|=| \mathfrak{G}|=|\overline{\mathcal{P}}_\mathfrak{G}|=2.\\
\end{array}
\]
By \cite[Theorem 3.3]{M},  no permutation having type in $\mathcal{T}\setminus\{[1,n-1]\}$ belongs
to a primitive subgroup. In particular this holds for all the types in $\mathcal{A}',
\mathcal{E}, \mathcal{F},\mathcal{G}.$ Moreover, by  Section \ref{case k=2}, it follows
immediately that the permutations in $\mathcal{A}'$ cannot belong to an imprimitive component
and so $\delta\supseteq \mathcal{P}_{\mathfrak{A}'},$ which gives $|\delta|\geq q-2.$ This
means that we need to force only $4$ further components to conclude.
Consider permutations having type in $\mathcal{E}$. Since these do not belong to any
subgroup in $\mathcal{P}_{\mathfrak{A}'}$, we need additional components to contain them and
those components must be intransitive or imprimitive. If $\delta$ contains neither
$\overline{S_2\wr S_{3q}}$ nor $\overline{S_{3q}\wr S_{2}},$ then
we need $q-1>4$ additional intransitive components and thus $|\delta|>q+2$, a contradiction.
So one of $\overline{S_2\wr S_{3q}}$ and $\overline{S_{3q}\wr S_{2}},$ belongs to $\delta$.
Denote this component by $I_2$ and note that we then need to force $3$ further components.

Similarly permutations having type in $\mathcal{F}$ do not belong to any of the subgroups in
$\mathcal{P}_{\mathfrak{A}'}$ or to $I_2$. If $\delta$ contains neither
$\overline{S_3\wr S_{2q}}$ nor $\overline{S_{2q}\wr S_{3}}$, then we need at least
$\frac{q-1}{2}>3$ (since $q\geq11$)  additional intransitive components to cover these permutations,
which is a contradiction. Thus one of $\overline{S_3\wr S_{2q}}$ or $\overline{S_{2q}\wr S_{3}}$,
denoted $I_3$, lies in $\delta$, and we need to force just $2$ more components.


Permutations of type in $\mathcal{G}$ do not belong to $I_2, I_3$, or any subgroup in
$\mathcal{P}_{\mathfrak{A}'}$. If $\delta$ contains neither $\overline{S_q\wr S_{6}}$
nor $\overline{S_{6}\wr S_{q}}$, then we need two additional intransitive components
to cover these permutations and the proof is complete. Thus we may assume that $\delta$ contains
one of $\overline{S_q\wr S_{15}}$ or $\overline{S_{15}\wr S_{q}},$ denoted $I_q$.
Finally, since no component in the subset $\mathcal{P}_{\mathfrak{A}'}\cup \{I_2,I_3,I_q\}$
of $\delta$ contains  the type $[1,n-1]$, it follows that we need an additional component for it,
and hence $|\delta|\geq q+2$.

\hspace{12cm}$\square$

\subsection{The symmetric group} 
Our tools for the symmetric group $S_n$ with $n$ even  seem too weak. We have no general argument even to show that at least the intransitive subgroups of $\mathcal{P}_{\mathfrak{A}}$ belong to any minimal basic set of maximal subgroups for $S_n$, even though computational evidence  with {\tt Magma} suggests that this is the case. In particular our computations gave no counterexamples to Conjecture \ref{conjecture-sym}.
For $n$ even, Conjectures \ref{conjecture-sym} and \ref{conjecture-alt} are linked: if Conjecture \ref{conjecture-alt} is true, then either also Conjecture \ref{conjecture-sym} is true, or $A_n$ belongs to each minimal basic set of $S_n.$ However, in our experience  $A_n$ does not play a significant role in the minimal basic sets of $S_n.$ As an example, a detailed argument of more than a page is needed to confirm that $\gamma(A_{30})=\gamma(S_{30})=g(30)=7$, and for space reasons we have not included it here.

\section{
Aknowledgements}
We thank F. Luca for helpful advice on the arithmetic section of the paper.

\end{document}